\documentclass[reqno]{amsart}

\usepackage{graphicx,pdfsync,amssymb, latexsym,amsfonts,amsbsy,color,subcaption,esint,mathtools,enumerate,nicefrac}

\usepackage{enumitem}
\setitemize{itemsep=0.5em}
\setenumerate{itemsep=0.5em}

\usepackage{hyperref}
\hypersetup{
colorlinks=true,
linkcolor=blue,
filecolor=blue,      
urlcolor=blue,
citecolor=blue,
}

\usepackage[utf8]{inputenc}
\usepackage[T1]{fontenc}

\usepackage{tikz,caption}
\usetikzlibrary{patterns}

\usepackage[top=1.1in, bottom=1in, left=1in, right=1in]{geometry}

\DeclareMathOperator{\Supp }{supp}

\DeclareMathOperator{\Lin }{lin} 
\DeclareMathOperator{\Cor }{cor} 
\DeclareMathOperator{\Osc }{osc} 
\DeclareMathOperator{\Acc}{acc}
\DeclareMathOperator{\D}{div}

\newtheorem{theorem}{Theorem}[section]
\newtheorem{lemma}[theorem]{Lemma}
\newtheorem{proposition}[theorem]{Proposition}

\newtheorem{remark}[theorem]{Remark}

\usepackage{times}
\usepackage{setspace}
\def \TT  {\mathbb{T}} 
\def \RR {\mathbb{R}}  
\def \NN {\mathbb{N}}  
\def \ZZ {\mathbb{Z}}  

\def \ep {\varepsilon}

\def \l {\lambda}

\def \ek {\mathbf{e}_{k}}

\def \bp {\mathbf{ \Phi}}
\def \bw {\mathbf{ W}}
\def \bwk {\mathbf{ W}_k}

\def \ep {\varepsilon}
\def \p {\partial}

\newcommand{\comment}[1]{}

\numberwithin{equation}{section}

\begin{document}

\title[Almost smooth nonuniqueness]{Extreme temporal intermittency in the linear Sobolev transport: almost smooth nonunique solutions}




\author{Alexey Cheskidov}
\address[Alexey Cheskidov]{Department of Mathematics, Statistics and Computer Science,
University of Illinois At Chicago, Chicago, Illinois 60607 and School of Mathematics, Institute for Advanced Study, 1 Einstein Dr., Princeton, NJ 08540, USA.}

\email{acheskid@uic.edu}

\author{Xiaoyutao Luo}

\address[Xiaoyutao Luo]{Department of Mathematics, Duke University, Durham, NC 27708 and School of Mathematics, Institute for Advanced Study, 1 Einstein Dr., Princeton, NJ 08540, USA.}

\email{xiaoyutao.luo@duke.edu}



\date{\today}

\begin{abstract}
In this paper, we revisit the notion of temporal intermittency to obtain sharp nonuniqueness results for linear transport equations. We construct divergence-free vector fields with sharp Sobolev regularity $L^1_t W^{1,p}$ for all $p<\infty$ in space dimensions $d\geq 2$ whose transport equations admit nonunique weak solutions belonging to $L^p_tC^k$ for all $p<\infty$ and $k\in \NN$. In particular, our result shows that the time-integrability assumption in the uniqueness of the DiPerna-Lions theory is sharp. The same result also holds for transport-diffusion equations with diffusion operators of arbitrarily large order in any dimensions $d \geq 2$.
 
\end{abstract}


\maketitle

\section{Introduction}

We consider the linear transport equation on the torus $\TT^d$ with $d\geq 2$,
\begin{equation}\label{eq:the_equation}
\begin{cases}
\partial_t \rho + u \cdot \nabla  \rho  =0 &\\
\rho|_{t=0} =\rho_0,
\end{cases}
\end{equation}
where $\rho : [0,T] \times \TT^d  \to \RR$ is a scalar density function and $u : [0,T] \times \TT^d  \to \RR^d$ is a given incompressible vector field, i.e. $ \D u =0 $. The linearity of the equation allows to prove the existence of weak solutions, even for very rough vector fields, that  satisfy the equation in the sense of distributions 
\begin{equation}\label{eq:def_weak_solutions}
\int_{\TT^d}\rho_0  \varphi(0,\cdot) \,dx = \int_0^T \int_{\TT^d}\rho (\p_t \varphi + u \cdot \nabla \varphi   ) \,dx dt   \quad \text{for all $\varphi \in C^\infty_c ([0,T) \times \TT^d)$} .
\end{equation}

In this paper, we focus on the issue of the uniqueness/nonuniqueness of weak solutions \eqref{eq:def_weak_solutions} for vector fields with Sobolev regularity. The celebrated DiPerna-Lions theory provides natural criteria for the uniqueness of the weak solutions for Sobolev vector fields:

\begin{theorem}[DiPerna-Lions~\cite{MR1022305}]
Let $p, q \in [1,\infty]$ and let $u \in L^1(0,T; W^{1, q} (\TT^d)) $ be a divergence-free vector field. For any $\rho_0 \in L^p(\TT^d)$, there exists a unique renormalized solution $\rho \in C( [0,T]; L^{p}(\TT^d) $  to \eqref{eq:the_equation}. Moreover, if
\begin{equation}\label{eq:DL_condition}
\frac{1}{p}  + \frac{1}{q} \leq 1,
\end{equation}
then this solution $\rho$ is unique among all weak solutions in the class $ L^\infty (0,T; L^p(\TT^d)$.
\end{theorem}

In recent years, there has been a growing interest~\cite{MR3884855,1902.08521,2003.00539,2004.09538} in showing the (possible) sharpness of the DiPerna-Lions condition \eqref{eq:DL_condition}, but so far the nonuniqueness constructions have not reached the full complement of \eqref{eq:DL_condition} in the class of $L^\infty_t L^p$ solutions. In this paper, we show that the time-integrability assumption in the DiPerna-Lions uniqueness theorem is sharp. More precisely,  
\begin{theorem}\label{thm:main_thm_short}
For any dimension $d\geq 2$ there exists a velocity vector field  $u\in L^1(0,T;W^{1,p}(\TT^d))$ for all $p<\infty$, such that the the uniqueness of \eqref{eq:the_equation} fails in the class 
$$ 
\rho \in \bigcap_{\substack{
  p<\infty  \\ k\in \NN}} L^p(0,T; C^k(\TT^d)).
$$
\end{theorem}

This result is proved by the convex integration technique,  brought to fluid dynamics by the pioneering work of De Lellis and Székelyhidi~\cite{MR2600877}, and has seen applications to the transport equation in \cite{MR3884855,1902.08521,2003.00539,2004.09538}. More details on the background and historical development will be discussed shortly. The key ingredient in the proof of Theorem~\ref{thm:main_thm_short} is the use of temporal intermittency,  introduced in our previous works \cite{2004.09538, 2009.06596, 2105.12117}. In particular, it improves our previous result \cite{2004.09538} in terms of the integrability in time of the solution $\rho$, and the spatial regularity of $u$ and $\rho$. Moreover, Theorem \ref{thm:main_thm_short} is sharp in the following two ways.
\begin{enumerate}
    \item  The vector field can not be $  L^1_t W^{1,\infty} $ for which any $L^1_{t,x} $ solution of \eqref{eq:the_equation} must coincide\footnote{For instance, by a duality argument and a commutator estimate.}  a.e.  with the Lagrangian solution;
    
    \item The density class can not have $L^\infty_t C^k$ regularity any for $k\in \NN$ due to the DiPerna-Lions condition \eqref{eq:DL_condition}.
\end{enumerate}

\subsection{Background and comparison}
While the classical method of characteristics implies the well-posedness of \eqref{eq:the_equation} for Lipschitz vector fields, for non-Lipschitz vector fields, the method of characteristics no longer applies, and the well-posedness of \eqref{eq:the_equation} becomes challenging. The renormalization theory of DiPerna-Lions \cite{MR1022305}  provides powerful well-posedness of \eqref{eq:the_equation} under suitable Sobolev regularity assumptions on the vector field, and the renormalized solutions are shown to be unique in the regime \eqref{eq:DL_condition}.

Since Aizenman's example~\cite{MR482853}, there have been examples of nonuniqueness at the Lagrangian level~\cite{MR2009116,MR3904158,MR2030717,MR3656475,1911.03271}, that is, constructions of vector fields whose flow
maps exhibit degeneration. However, for a long time, the existence of nonunique (Eulerian) weak solutions of \eqref{eq:the_equation} for divergence-free Sobolev vector fields $u \in L^1_t W^{1,p}$ was unknown. To our knowledge, the first Eulerian construction of nonuniqueness was obtained in \cite{MR3393180} using the framework of \cite{MR2600877} for bounded vector fields.

Inspired by the spatially intermittent construction in \cite{MR3898708}, the breakthrough result~\cite{MR3884855} of Modena and Sz\'ekelyhidi gave the first example of a Sobolev vector field with nonunique weak solutions to \eqref{eq:the_equation} and led to a lot of interest in improving nonuniqueness constructions to larger functional classes. Below we list the regimes where the nonuniqueness has been achieved:

\begin{enumerate}
    \item \cite{MR3884855,MR4029736} (Modena and Sz\'ekelyhidi): $\rho \in C_t L^p  $ when $u \in C_t W^{1,q} $ for $\frac{1}{p}  + \frac{1}{q} > 1 + \frac{1}{d-1}$ and $d\geq 3$. 
    
    \item \cite{1902.08521} (Modena and Sattig): $\rho \in C_t L^p  $ when $u \in C_t W^{1,q} $ for $\frac{1}{p}  + \frac{1}{q} > 1 + \frac{1}{d}$.
    
\item \cite{2003.00539} (Bru\`e, Colombo, and De Lellis): positive\footnote{Well-posedness for positive $\rho$ can go beyond the DiPerna-Lions range, see \cite[Theorem 1.5]{2003.00539}} $\rho \in C_t L^p  $ when $u \in C_t W^{1,q}$ for $\frac{1}{p}  + \frac{1}{q} > 1 + \frac{1}{d}$. 

\item \cite{2004.09538} (Cheskidov and Luo): $\rho \in L^1_t L^p  $ when $u \in L^1_t W^{1,q} $ for $\frac{1}{p}  + \frac{1}{q} > 1$ and $d\geq 3$.
\end{enumerate}

In summary, in the class of $L^\infty_t L^p$ densities, the nonuniqueness has been achieved in the regime $\frac{1}{p}  + \frac{1}{q} > 1 + \frac{1}{d}$, while nonuniqueness in the regime $\frac{1}{p}  + \frac{1}{q} > 1  $ is possible if one settles for $L^1_t L^p$ densities. However, it was not known whether $\frac{1}{p}  + \frac{1}{q} = 1  $ is still the critical threshold for $L^1_t L^p$ densities.

Our main goal here is to show that the DiPerna-Lions scaling $\frac1p + \frac1q=1$ becomes irrelevant once the time integrability of $\rho$ is slightly weakened. In particular, Theorem~\ref{thm:main_thm_short} follows from the  following  convex integration construction.

\begin{theorem}\label{thm:main_thm}
Let $d\geq  2$, $\ep>0$, and $N\in \NN $. Let $ \widetilde{\rho} \in C^\infty  (   \TT^d\times \RR )$ be such that $ \Supp_t  \widetilde{\rho} \subset (0,T)$ and $\fint_{\TT^d} \rho(x,t)\, dx =0$ for all $t\in \RR $.

Then there exist a divergence-free vector field $u: \TT^d \times [0,T] \to \RR^d$ and a density $\rho : [0,T] \times \TT^d \to \RR $ such that all of the following hold.

\begin{enumerate}
	\item   $u\in L^1(0,T;W^{1,p}(\TT^d))$  and $\rho \in L^p(0,T; C^k(\TT^d))$ for all $1\leq p<\infty$ and $k\in \NN$.

	\item $\rho u \in L^1([0,T] \times \TT^d)$ and $(\rho, u)$ is a weak solution to \eqref{eq:the_equation} in the sense of \eqref{eq:def_weak_solutions}.
	
	\item The  deviation of $\rho$ in $C^N(\TT^d) $ norm is small: $ \|\rho - \widetilde{\rho} \|_{L^N_t C^N} \leq \ep$.  
	
	\item $ \rho$ has a compact temporal support:   $  \Supp_t \rho \subset \Supp_t \widetilde{\rho} $.
	
\end{enumerate}
\end{theorem}

\begin{remark}
\hfill
\begin{enumerate}
\item Here our initial data is always zero and attained in the classical sense. It is also easy to show that the obtained solution $\rho$ is continuous in time in the sense of distributions (see Lemma 7.7 in \cite{2004.09538} for details).

\item Theorem \ref{thm:main_thm} continues to hold for the transport-diffusion equation with a parabolic regularization $\Delta^m \rho$ of arbitrary order in the same regularity classes $(\rho, u) \in L^p_t C^k \times L^1_t W^{1,p}$, see Theorem~\ref{thm:zero_viscosity_limit}.  To our knowledge, this is the first example of a PDE where parabolic regularization does not provide any additional rigidity for the uniqueness of a class of weak solutions.

\item The nonunique solutions $\rho$ must change their signs--- it is known since \cite{1812.06817} and \cite{2003.00539} that any sign-definite solution  $ \rho \in L^1_{t,x}$ of  $L^1_t W^{1,d+} $ vector fields is Lagrangian.

\item By the linearity of \eqref{eq:the_equation}, for any initial data $\rho_0 \in L^p(\TT^d )$, the constructed vector field gives nonunique solutions in the class $\rho \in L^q_t L^p $ for any $q<\infty$. Indeed, one can add the constructed solution on top of the renormalized solution associated to $  \rho_0$.

\end{enumerate}
\end{remark}
\subsection{Strategy of the proof}
We conclude with some final remarks on the proof. As said, we used the convex integration technique brought to fluid dynamics by the pioneering work of De Lellis and Székelyhidi \cite{MR2600877}. The groundbreaking technique of \cite{MR2600877} resulted in breakthrough works in the fluids community over the last decade, and we refer readers to \cite{MR2600877,MR3090182,MR3866888,1701.08678,MR3884855,MR3898708} and references therein for a complete account.

The construction follows the same framework of temporal intermittency in our previous work \cite{2004.09538}.  A key difference is the regularity $L^p_t C^k$ for the density, which requires extreme intermittency in time when progressing to high frequencies. Since the density does not enjoy any ``reasonable'' $L^\infty_t$ regularity, from the duality $u \rho \in L^1_{t,x}$ we can gain a surprising regularity of almost $L^1_t$ Lipschitz of the vector field. As in \cite{2004.09538}, this extreme temporal intermittency necessitates the use of stationary building blocks, as otherwise, the error produced by the large acceleration of the density becomes insurmountable with the non-Lipschitzness of the vector field, cf. Lemma \ref{lemma:parameters} below. Once extreme intermittency in time is achieved, a little deduction of the time regularity of the density from $L^\infty_t $ to $L^p_t$ allows us to gain essentially infinitely many derivatives in space for the density.

Finally, since the density enjoys essentially infinite many derivatives in space, the same construction also holds for transport-diffusion equations with diffusion operators of arbitrarily large order in any dimension $d \geq 2$. Surprisingly, even in dimension $d=2$ a diffusion of arbitrarily high order is not able to provide uniqueness for this class of weak solutions.

\subsection*{Organization}
The rest of the paper is organized as follows.
\begin{itemize}
\item We prove the main theorem stated in the introduction in Section \ref{sec:proof} by assuming Proposition~\ref{prop:main_prop}, whose proof is the main content of this paper.

\item In Section~\ref{sec:temporal_BB_perturbations}, we first introduce temporal intermittency into the construction, essential for our scheme. Next, we recall Mikado densities and Mikado flows as spatial building blocks. Finally, we use these temporal and spatial building blocks to define the density and velocity perturbations.

\item In Section~\ref{sec:Estimates_on_perturbation}, we first specify the oscillation and concentration parameters and obtain estimates on the velocity and density perturbations claimed in Proposition~\ref{prop:main_prop}.

\item Section~\ref{sec:New_defect_field} is devoted to deriving the new defect field and its estimates, finishing the proof of Proposition~\ref{prop:main_prop}.

\item In Section \ref{sec:transport-diffusion},
we show that the same nonuniqueness holds for transport-diffusion equations with arbitrarily high diffusion as well. 

\item  In Appendix~\ref{sec:append_tools}, we recall some (now standard) technical tools in convex integration, namely the improved H\"older inequalities and antidivergence operators.  
\end{itemize}

\subsection*{Acknowledgement}
AC was partially supported by the NSF grant DMS--1909849. XL is partially supported by the NSF grant DMS-1926686. The authors are grateful for the hospitality of the Institute for Advanced Study where a part of this work was done.

\section{The main proposition and proof of Theorem \ref{thm:main_thm}}\label{sec:proof}

\subsection{Notations}
Throughout the paper, we fix the spatial domain $\TT^d = \RR^d/\ZZ^d$, identified with a periodic box $[0,1]^d$. Average over $\TT^d$ is denoted by $\fint f = \int_{\TT^d} f$. Functions on $\TT^d$ are identified as periodic ones in $\RR^d$, and we say $f$ is $\sigma^{-1}\TT^d$-periodic if 
$$
f(x + \sigma k) = f(x) \quad \text{for any $k\in \ZZ^d$}.
$$

Spatial Lebesgue norms are denoted by $ \| \cdot  \|_{L^p} = \| \cdot  \|_{L^p(\TT^d)} $, while we write $\| \cdot  \|_{L^p_{t,x}} $ for Lebesgue norms taken in the space-time domain $\TT^d\times [0,T]$. If a function $f$ is time-dependent, we write $ \|f(t) \|_{L^p}, $
to indicate that the spatial norm is taken at a time slice $t\in [0,T]$. For a Banach space $X$, we use the notation $\| \cdot \|_{L^p_t X}$ to denote the norm on Bochner spaces $L^p([0,T]; X)$, such as $\| \cdot  \|_{L^1_t W^{k,p}}$ and $\| \cdot  \|_{L^p_t C^{k}}$.

The differentiation operations such as $\nabla$, $\Delta$, and $\D$ are meant for differentiation in space only.

We use the notation $X \lesssim Y$ which means $X \leq C Y$ for some constant $C >0$.  The notation $X \sim Y$ means both $X \lesssim Y$ and $Y \lesssim X$ at the same time.

\subsection{Continuity-defect equation}
As in \cite{MR3884855}, we consider the continuity-defect equation to obtain approximate solutions to the transport equation:
\begin{equation}\label{eq:defect_equation}
\begin{cases}
\p_t \rho + \D(\rho u) = \D R\\
\D  u =0,
\end{cases}
\end{equation}
where $R: [0,T] \times \TT^d \to \RR^d$ is called the defect field. In what follows, a triple $(\rho, u, R)$ will denote a smooth solution to \eqref{eq:defect_equation}. Recall that the for $f\in L^1_{t,x} $ function, its temporal support $\Supp_t f $ is the closure of the set $  \{ t\in [0,T]: |f(x,t)| > 0  \quad a.e.\, \, x\in \TT^d  \} $.

We now state the main proposition of the paper and use it to prove Theorem \ref{thm:main_thm}.
\begin{proposition}\label{prop:main_prop}
Let $d \geq  2$. There exist a universal constant $M>0$ such that the following holds. 

Suppose $(\rho, u, R) $ is a smooth solution of \eqref{eq:defect_equation} on $[0,1]$ such that $\Supp_t R \subset (0,1)$. Then for any $  1\leq p \in \NN $ and any $0<\delta<1/2$, there exists another smooth solution $(\rho_1, u_1, R_1)$ of \eqref{eq:defect_equation} on $[0,1]$ such that the density perturbation $\theta: = \rho_1 -\rho$ and the vector field perturbation $ w = u_1 -u$ satisfy
\begin{enumerate}
    \item Both $\theta$ and $ w  $ have zero spacial mean and
\begin{equation} \label{eq:support_of_theta}
\Supp_t \theta \subset  \Supp_t R.
\end{equation}

\item $\theta$ and $ w  $ satisfy the estimates
\begin{align}
\| \theta \|_{L^p_t C^p} &\leq \delta,  \label{eq:main_prop_0}\\
 \| w  \|_{L^1_t W^{1, p }}  &\leq \delta,
 \label{eq:main_prop_1}\\
\| \theta w + \theta u + \rho w   \|_{L^1_{t,x}}  & \leq  M \|R \|_{L^1_{t,x}}. \label{eq:main_prop_2}
\end{align}

\item  The new defect field $R_1$ satisfies
\begin{align} \label{eq:support_of_R_1}
  \Supp_t R_1  \subset \Supp_t R,
\end{align} 
and  the estimate
\begin{align}
\|R_1 \|_{L^1_{t,x}} \leq \delta. 
\end{align}

\end{enumerate}

\end{proposition}

\subsection{Proof of Theorem \ref{thm:main_thm} }
\begin{proof}\label{subsec:proof_main_thm}
We assume $T=1$ without  loss of generality. We will construct a sequence $(\rho_n, u_n , R_n) $, $n=1,2\dots$ of solutions to \eqref{eq:defect_equation} as follows. 
For $n=1$, we set 
\begin{align*}
\rho_1(t) &:= \widetilde{\rho},  \\ 
u_1(t) &:= 0, \\
R_1(t) &:=  \mathcal{R} \big( \p_t \widetilde{\rho} \big),
\end{align*}
where $\mathcal{R}=\Delta^{-1}\nabla $ is the inverse divergence in Appendix \ref{sec:append_tools}. Then $ (\rho_1, u_1 , R_1)  $ solves \eqref{eq:defect_equation} trivially by the constant mean assumption on $\widetilde{\rho} $.

Next, we apply Proposition \ref{prop:main_prop} inductively to obtain $ (\rho_n, u_n , R_n)$ for $n=2,3\dots$ as follows.
Given $(\rho_n, u_n,R_n) $, we apply Proposition \ref{prop:main_prop} with parameters \[
  p_n  =N2^n, \qquad  \delta_n=\varepsilon 2^{-n},
\]
to obtain a new triple $(\rho_{n+1}, u_{n+1},R_{n+1}) $. Then the perturbations $\theta_n : = \rho_{n+1} - \rho_{n}$, $w_n : = u_{n+1} - u_{n}$ and the defect fields $R_n$ satisfy
\begin{subequations}
\begin{align}
  &\|\theta_n \|_{L^{p_n }_t C^{p_n }} \leq   \delta_n, 
 \qquad   \| w_n \|_{L^1_t W^{1,p_n  }} \leq  \delta_n,  \label{eq:proof_main_theorem_1a}\\
 &  \|R_{n+1}\|_{L^1_{t,x}} \leq \delta_n, \label{eq:proof_main_theorem_1b}\\
&\| \theta_n w_n + \theta_n u_n + \rho_n w_n   \|_{L^1_{t,x}}\leq M\|R_n\|_{L^1_{t,x}},\label{eq:proof_main_theorem_1c}
\end{align}
\end{subequations}
for all $n =1,2\dots$.
In addition, due to \eqref{eq:support_of_R_1} and \eqref{eq:support_of_theta}, we have
\begin{equation} \label{eq:supp_of_theta_n}
 \Supp_t \theta_n \subset \Supp_t \widetilde \rho, \qquad \forall n\in \mathbb{N}.
\end{equation}
Hence by \eqref{eq:proof_main_theorem_1a} there exists $(\rho, u) \in L^p_t C^p \times L^1_t W^{1,p}$ for all $p \in \NN $  such that
\begin{align}
 \rho_n   \rightarrow  \rho \quad \text{in } L^p_t C^p  \quad \text{and}\quad
 u_n  \rightarrow  u  \quad \text{in } L^1_t W^{1, p },\quad   \forall p \in \NN   .
\end{align}
Moreover,
 $\Supp_t \rho \subset \Supp_t \widetilde \rho$ due to \eqref{eq:supp_of_theta_n}. Since $p_n \geq N$ and the time interval is of length $1$,
 $$
\|\rho - \widetilde{\rho} \|_{L^N_t C^N} \leq \sum_{n\geq 1} \|\theta_n \|_{L^N_t C^N} \leq \sum_{n\geq 1} \|\theta_n \|_{L^{p_n }_t C^{p_n }} \leq   \ep.
$$

It remains to show $(\rho, u) $ is a weak solution.  We first prove that $\rho u \in L^1_{t,x}$ and $ \rho_n u_n  \rightarrow  \rho u$  in $L^1_{t,x}$. Using \eqref{eq:proof_main_theorem_1c},
\begin{equation} \label{eq:product_theta_n_u_n}
\|\rho_{n+1} u_{n+1} - \rho_{n} u_{n}\|_{L^1_{t,x}}  \leq   M\delta_{n-1} \quad \text{for $n \geq 2$. }
\end{equation}
Thus the sequence $\rho_n u_n$ is Cauchy in $L^1_{t,x}$ and consequently there is $G \in L^1_{t,x}$ such that $ \rho_n u_n  \rightarrow  G$   in  $  L^1_{t,x} $. Now we claim that $G = \rho u$. Thanks to \eqref{eq:product_theta_n_u_n}, passing to subsequences and dropping subindices, we get $\rho_n \to \rho $ and $u_n \to u$ a.e. in $\TT^d \times [0,1]$. So $\rho_n u_n \to G$ a.e. in $\TT^d \times [0,1]$ and hence $\rho u  =G$  and $ \rho_n u_n  \rightarrow  \rho u$  in $L^1_{t,x}$. Since in addition $R_n  \rightarrow  0  $
 in  $ {L^1_{t,x}}$ by \eqref{eq:proof_main_theorem_1b}, it is standard to show that $(\rho, u) $ is a weak solution to \eqref{eq:the_equation}

\end{proof}

\section{Temporal intermittency, building blocks, and perturbations} \label{sec:temporal_BB_perturbations}
The rest of the paper is devoted to the proof of Proposition~\ref{prop:main_prop}. In this section, we introduce the temporal and spatial building blocks and use them to define the density and velocity perturbations.

\subsection{Summary of parameters}
Given arbitrarily large $p\in \NN$ as in the statement of Proposition~\ref{prop:main_prop}, we will fix three exponents: $r>1$ very close to $1$,  $0<\alpha \ll 1$, and $0<\gamma \ll 1$ in Lemma \ref{lemma:parameters} below. These exponents are used to define three large parameters: concentrations $\kappa ,\mu \geq 1$ and oscillation $ \sigma \in \NN$. These three large parameters satisfy the hierarchy $1\ll \sigma  \ll \mu \ll \kappa  $, whose meaning will be made precise in Section \ref{sec:Estimates_on_perturbation}, but their exact values will be fixed at the end depending on the given solution $(\rho,u,R) $.

\subsection{Temporal functions \texorpdfstring{$\widetilde{g}_k $}{tilde gk} and \texorpdfstring{$g_k $}{gk}}

We start with defining the intermittent oscillatory functions $\widetilde{g}_k $ and $g_k$ that lie at the heart of our scheme. First, we choose a profile function $\widetilde G  \in C^\infty_c ((0,1))$ such that 
\begin{equation}\label{eq:def_of_tilde_G}
\int_{[0,1]} \widetilde G^2  \,dt =1, \qquad \int_{[0,1]} \widetilde G  \,dt =0, \qquad  \|\widetilde G\|_{L^\infty} \leq 2,   
\end{equation}
and for $k=1 \dots d$ define $ G_k$ to be the $1$-periodic extension of $\widetilde G(\kappa (t-t_k))$, where $t_k \in [0,1 ]$ are chosen so that $ G_k$ have disjoint supports for different $k$. In other words, $  G_k(t) =   \sum_{n \in \ZZ} \widetilde G (n + \kappa (t - t_k) )$. We will refer to $\kappa \geq 1$ as the temporal concentration parameter.

Next, for a large oscillation parameter $\sigma \in \NN$ and a small exponent $0<\alpha<1$ to be fixed later, we define $\sigma^{-1}$-periodic functions
\begin{align}
\widetilde{g}_{k }(t) = \kappa^{\alpha}  G_{ k  }(\sigma t), \qquad g_{k }(t) = \kappa^{1-\alpha}  G_{ k }( \sigma t).
\end{align}
We will use $ \widetilde{g}_{\kappa }$ to oscillate the densities $\bp_k$, and $ g_{\kappa } $ to oscillate the vectors $ \bw_k$, defined in the following section. Note that by \eqref{eq:def_of_tilde_G}
\begin{equation} \label{eq:g_theta_g_w_interaction}
\int_{[0,1]} \widetilde{g}_k    g_k \, dt = 1,
\end{equation} 
and by definitions of $\widetilde{g}_k     $ and $g_k$,
\begin{equation} \label{eq:estimates_g}
\|  \widetilde{g}_k \|_{L^q([0,1])} \sim \kappa^{\alpha - \frac{1}{q} }, \qquad
 \|  \widetilde{g}'_k \|_{L^q([0,1])} \sim (\kappa \sigma )\kappa^{\alpha - \frac{1}{q} }, \qquad
\|  {g}_k \|_{L^q([0,1])} \sim \kappa^{1-\alpha  - \frac{1}{q} }.
\end{equation}

\subsection{Temporal correction function \texorpdfstring{$h_k $}{hk}}
Now we define a $\sigma^{-1}$-periodic function $  h_k :\RR \to \RR$ by
\begin{equation}
 h_{k} (t):= \sigma \int_0^t (\widetilde{g}_k {g}_{k } -1) \, d\tau,
\end{equation}
so that
\begin{equation}\label{eq:def_h_mu_t}
\sigma^{-1}\p_t  h_{k} =   \widetilde{g}_{k}   {g}_{k}  -1.
\end{equation}
Thanks to \eqref{eq:g_theta_g_w_interaction}, $ h_{k}$  is well-defined and satisfies the estimate
\begin{align}\label{eq:def_h_mu_t_L_1}
\|  h_{k} \|_{L^\infty[0,1]} \leq 1.
\end{align}
The function $h_k$ will be used to design the temporal corrector $\theta_o$ in \eqref{eq:def_theta_o_2}.

\subsection{Mikado densities and flows}
In this subsection, we recall the spatial building blocks for our convex integration construction, Mikado densities and Mikado flows introduced in \cite{MR3614753} and \cite{MR3884855}. These are periodic objects supported on pipes with a small radius. Note that we do not require them to have disjoint supports in space--- each Mikado object will be coupled with a temporal function $ \widetilde{g}_k  $ or $ {g}_k $ to achieve disjoint supports in space-time.

For $k=1, \dots, d$, we denote each standard Euclidean basis vector $\ek = (0,\dots,1,\dots,0)$. For any $x\in \RR^d$ and $k=1, \dots, d$, $x'_k \in \RR^{d-1}$ denotes the vector  $x'_k = (x_1, x_2,\dots, x_{k-1}, x_{k+1},\dots,x_d)$.

Let $d \geq 2$ be the spatial dimension. We fix a vector field $\Omega \in C_c^\infty(\RR^{d-1}  )  $ and a scalar density $ \phi \in C_c^\infty(\RR^{d-1} )$ such that
\begin{equation}\label{eq:def_normlazied_phi}
\Supp \Omega \subset (0,1)^{d-1} \quad  \text{and}  \quad \D \Omega = \phi \quad  \text{and} \quad    \int_{\RR^{d-1}} \phi^2 = 1.
\end{equation}

For each $k = 1,\dots, d$, we define the non-periodic Mikado objects
\begin{equation}\label{eq:non_per_Mikados}
\begin{aligned}
\widetilde \Phi_k (x ) &=   \phi (\mu x_k'),\\
\widetilde \Omega_k (x )& = \mu^{-1}\Omega( \mu x_k'),\\
\widetilde W_k (x ) &= \mu^{ d-1 } \phi  (\mu x_k') \ek,
\end{aligned}
\end{equation}
and then define the $1$-periodic objects $ \Omega_k: \TT^d \to \RR^d$, $\Phi_k: \TT^d \to \RR$ and $W_k: \TT^d \to \RR^d$ as the $1$-periodic extensions of \eqref{eq:non_per_Mikados}, and then rescale them by a large oscillation factor $\sigma \in \NN$:
\begin{equation}\label{eq:bold_Mikados}
\begin{aligned}
 \bp_k(x ) & =    \Phi_k(\sigma x ), \qquad
 \mathbf{\Omega}_k(x ) & = \Omega_k(\sigma x ), \qquad 
 \bwk(x) & =   W_k( \sigma x). 
\end{aligned}
\end{equation} 

We now summarize the properties of the constructed building blocks $ \mathbf \Omega_k $, $\mathbf\Phi_k $, and $\mathbf W_k$ in the following theorem.

\begin{theorem} \label{thm:Mikados}
For all $\sigma \in \NN$, and $\mu \geq 1$,  the density $\mathbf{ \Phi}_k $, potential $\mathbf{\Omega}_k$, and  vector field $\bwk$ defined by \eqref{eq:bold_Mikados} satisfy the following for every $k=1\dots d$.

\begin{enumerate}
\item $\mathbf W_k :\TT^d \to \RR^d $,  $\mathbf\Phi_k:\TT^d  \to \RR $, and $\mathbf\Omega_k:\TT^d  \to \RR^d$ are smooth functions and have zero mean on $\TT^d$.
    \item  $\D \bwk = \D(\bp_k \bwk )=0$ and the density $\mathbf{ \Phi}_k $ is the divergence of the potential $ \sigma^{-1} \mathbf{\Omega}_k$,
\begin{align}\label{eq:Div_Omega_k_close_Phi_k}
  \D \mathbf{\Omega}_k  =     \sigma \mathbf{ \Phi}_k. 
\end{align}

\item For any $1\leq p\leq \infty $ and $s \geq 0$,
\begin{subequations}
\begin{align}
\| \mathbf{\Omega}_k\|_{L^\infty_t L^p} & \lesssim    \mu^{ -1 -\frac{ d-1 }{p}}, \label{eq:BB_Bounds_Omega}\\
 \|\bp_k \|_{  W^{s,p} } & \lesssim   (\sigma \mu)^s  \mu^{- \frac{ d-1 }{p} }, \label{eq:BB_Bounds_Phi}\\
 \|\bwk \|_{W^{s,p}} & \lesssim   (\sigma \mu)^s  \mu^{ (d- 1) \left(1-\frac1p\right)}.\label{eq:BB_Bounds_W}
\end{align}
\end{subequations}

    \item The following identity holds 
\begin{equation} \label{eq:(PtimesW)_self_interaction}
\fint_{\TT^d} \bp_k(x ) \bwk(x) \, dx = \ek .
\end{equation}

\end{enumerate}
\end{theorem}
\begin{proof}
The first two points are direct consequences of the definitions while the last point follows from \eqref{eq:def_normlazied_phi}.  

When $s=0$, the bounds \eqref{eq:BB_Bounds_Omega}--\eqref{eq:BB_Bounds_W} follow from the small supports of the non-periodic objects $ \widetilde \Phi_k , \widetilde \Omega_k  ,\widetilde W_k $: the support set is a cylinder of radius $\sim \mu^{-1}$ and length $1$. The general case $s>0$ can be obtained by interpolation between the cases $s\in \NN$.

\end{proof}

\subsection{Density and velocity perturbations} \label{subsec:theta_w_definition}

Here we define perturbations $(\theta , w)$ given a defect field $R$ as in Proposition~\ref{prop:main_prop}.

Recall that the concentration parameters $\mu,\kappa \geq 1 $ and the oscillation parameter $\sigma \in \NN$ introduced so far will be specified in Lemma \ref{lemma:parameters} below.  The velocity perturbation  is   defined by
\begin{align}\label{eq:def_w}
w    &:=         \sum_{1\leq k \leq d}   g_{k}   \bw_k.
\end{align}

For the density perturbation, first we decompose the defect field 
\begin{equation}
R(x,t)  = \sum_{1\leq k \leq d} R_k(x,t) \mathbf{e}_k,
\end{equation}
where $\ek$'s are the standard Euclidean basis as before. We define the density  perturbation as the sum of the zero-mean projection of the principal part and a small oscillation correction:
\[
\theta =  \mathbb{P}_{\ne 0} \theta_p+\theta_o, 
\]
where $\mathbb{P}_{\ne 0} f = f -\fint f$ is the projection removing the spatial mean, and 
\begin{align}
\theta_p  &:=    -   \sum_{1\leq k \leq d} \widetilde{g}_{k}  R_k   \bp_k,  \label{eq:def_theta_p} \\
\theta_o  &  = \sigma^{-1}  \D \sum_{1\leq k \leq d}  h_k R_k    \ek. \label{eq:def_theta_o_2}
\end{align}

Note that $\D w=0$ for all $t$ since $\bw_k$ is divergence-free, which also implies that
\[
\D\big([\mathbb{P}_{\ne 0}\theta_p] w \big) = \D(\theta_p w ). 
\]

By definitions,  $\Supp_t \theta \subset  \Supp_t R$ as required in \eqref{eq:support_of_theta} of Proposition~\ref{prop:main_prop}.

\section{Estimates of the density and velocity perturbations} \label{sec:Estimates_on_perturbation}

The goal of this section is to obtain estimates \eqref{eq:main_prop_0}, \eqref{eq:main_prop_1}, and \eqref{eq:main_prop_2} on $\theta$ and $w$ claimed  in Proposition~\ref{prop:main_prop}.

\subsection{Choice of parameters}

Now we specify all the oscillation and concentration parameters in the perturbation as explicit powers of a large frequency number $\lambda>0$ that will be fixed in the end.

\begin{enumerate}
    \item Oscillation $\sigma  \in \NN $:
    \begin{align*} 
    \sigma & = \lceil \l^{ 2\gamma } \rceil.  
    \end{align*}
    Without loss of generality, we only consider values of $\l$ such that $\sigma = \l^{2 \gamma } \in \NN $ in what follows.

    \item Concentration $\kappa   ,\mu \geq 1$:
    \begin{align*} 
    \mu &= \l\\
    \kappa &=    \l^{\frac{d  - 2 \gamma  }{\alpha}}.
    \end{align*}

\end{enumerate}

\begin{lemma} \label{lemma:parameters}
For any $ p\in \NN$, there exist constants $\alpha>0$,   $0<\gamma < 1/4$, and $r>1$ such that the following holds,
\begin{align}
    (\sigma\mu)^p\kappa^{\alpha - \frac1p} &\leq \l^{-\gamma} \qquad (\theta_p \in L^p_t C^p),\label{eq:condition_theta_p_Lq_Linf}\\
  \kappa^{-\alpha}    (\sigma \mu)^1 \mu^{(d-1)(1 -\frac{1}{p})} &\leq \l^{-\gamma} \qquad (w  \in L^1_t W^{1,p}),\label{eq:condition_w_L1_W_1p}\\
    \kappa^{ \alpha}   \mu^{-1 -\frac{d-1}{r}} &\leq \l^{-\gamma} \qquad (\text{acceleration error}).\label{eq:condition_acc_erro}
\end{align}
 
\end{lemma}

\begin{proof} 

We first fix $\gamma>0$. Condition \eqref{eq:condition_w_L1_W_1p} in terms of power of $\l$ reads
\[
  \frac{d-1}{p}     \geq  5 \gamma.
\]
Since $p<\infty$, this condition is satisfied for $0< \gamma< \frac{1}{4} $ sufficiently small.
Expressing \eqref{eq:condition_theta_p_Lq_Linf} in terms of power of $\l$ gives $ \alpha \leq \frac{1}{p} \frac{d- 2 \gamma}{ ( 2p\gamma + p +d  -  \gamma)}$. Since $0<\gamma< 1/4$, this condition on $\alpha $ is automatically satisfied when $\alpha< \frac{d -1/2}{  2 p^2 + 2 d p  } $. We then fix $\alpha>0$ according to this condition.

For condition~\eqref{eq:condition_acc_erro}, taking $r=1$, the left-hand side becomes
\[
  \l^{d - 2\gamma -1 -d +1  } =\l^{-2 \gamma}  .
\]
Therefore, by continuity, \eqref{eq:condition_acc_erro} holds for $r>1$ close enough to $1$.
\end{proof}

\subsection{Estimates for the perturbations}

In what follows, $C_R$ will stand for a large constant that only depends on the triple $(\rho,u,R)$ provided as the input by Proposition~\ref{prop:main_prop}. It is important that $C_R$ can \textbf{never} depend on the free parameters $\sigma, \mu $ and $\kappa$ in the building blocks that we used to define $ \theta$ and $w$. 

\begin{lemma}[Estimate of $\theta$]\label{lemma:theta_estimate}
The density perturbation $\theta$ satisfies
\begin{equation*}
\|\theta  \|_{L^p_t C^p} \leq  C_R \lambda^{-\gamma}.
\end{equation*}
 
\end{lemma}
\begin{proof}
For the principle part $\theta_p$, since the space $C^p(\TT^d)$ is an algebra, using H\"older's inequality, \eqref{eq:estimates_g}, and \eqref{eq:BB_Bounds_Phi}, we obtain
\[
\begin{split}
\|\theta_p  \|_{L^p_t C^p} & \leq  \sum_{1\leq k \leq d} \|\widetilde{g}_{k}\|_{L^p} \|R_k\|_{L^\infty_t C^p }   \|\bp_k\|_{L^\infty_t C^p } \\
&\leq C_R    (\sigma\mu)^p \kappa^{\alpha - \frac{1}{p} }   \\
& \leq   C_R \lambda^{-\gamma},
\end{split}
\]
where the last inequality holds  due to condition \eqref{eq:condition_theta_p_Lq_Linf}.

For the temporal corrector $\theta_0$ defined in \eqref{eq:def_theta_o_2}, by H\"older's inequality and \eqref{eq:def_h_mu_t_L_1} we have
\begin{equation} \label{eq:theta_o_estimates}
\begin{split}
\| \theta_o \|_{L^\infty_t C^p} & \leq \sigma^{-1}  \sum_{1\leq k \leq d}\|  h_k \|_{L^\infty([0,1])}  \big\|     \D (R_k  \ek)\big\|_{L^\infty_t C^p} \\
&\leq C_R \sigma^{-1},  
\end{split}
\end{equation}
and the final bound holds by definition of $\sigma$.
\end{proof}

\begin{lemma}[Estimate on $w$]\label{lemma:w_estimate}
The velocity perturbation $w$ satisfies
\begin{align*}
\| w  \|_{L^1_t W^{1,p}} \lesssim \lambda^{-\gamma}.
\end{align*}
\end{lemma}
\begin{proof}
Using H\"older's inequality, \eqref{eq:estimates_g}, and \eqref{eq:BB_Bounds_W}, we obtain
\[
\begin{split}
\|  w   \|_{L^1_t W^{1,p}} & \leq      \sum_{1\leq k \leq d}   \|g_{k}\|_{L^1}  \|\bw_k\|_{W^{1,p }}\\
&\lesssim     \kappa^{-\alpha}    (\sigma \mu) \mu^{(d-1)(1 - \frac{1}{p})}.
\end{split}
\]
The conclusion holds thanks to \eqref{eq:condition_w_L1_W_1p}.

\end{proof}

\begin{lemma}[Estimate on $\theta w$]\label{lemma:theta_w_estimate}
The following estimate holds:
\begin{align*}
\| \theta w + \theta u + \rho w   \|_{L^1_{t,x}} \lesssim \|R\|_{L^1_{t,x}} + C_R\lambda^{-\gamma}.
\end{align*}
\end{lemma}
\begin{proof}
                       Taking the $L^1$ norm in space and using Lemma~\ref{lemma:improved_Holder} and the fact that $\bp_k \bw_k$ is $\sigma^{-1}\TT^d$-periodic in space, we obtain
\[
\begin{split}
\| \theta(t) w (t) \|_{L^1} & \leq  \sum_{1\leq k \leq d}   |\widetilde g_k(t) g_{k}(t)|  \|R_k (t) \bp_k \bw_k\|_{L^1}\\
& \lesssim  \sum_{1\leq k \leq d}   |\widetilde g_k(t) g_{k}(t)|  \|\bp_k \bw_k\|_{L^1}\left(\|R_k(t)\|_{L^1} + \sigma^{-1}\|R_k(t) \|_{C^1} \right)\\
& \lesssim  \sum_{1\leq k \leq d}   |\widetilde g_k(t) g_{k}(t)|  \left(\|R_k (t)\|_{L^1} + \sigma^{-1}\|R_k\|_{C^1_{t,x}} \right),
\end{split}
\]
where we used $\|\bp_k \bw_k\|_{L^\infty_tL^1_{x}}\sim  1$ by \eqref{eq:BB_Bounds_Phi} and \eqref{eq:BB_Bounds_W} at the last step.
Now taking the $L^1$ norm in time, using Lemma~\ref{lemma:improved_Holder} together with $\sigma$-periodicity of $g_k(t) g_{k}(t)$ and the smoothness of $t \mapsto \|R_k (t)\|_{L^1}  $, and recalling that $\|g_k g_{k}\|_{L^1}\sim 1$, we arrive at
\[
\begin{split}
\| \theta w  \|_{L^1_{t,x}}
&\lesssim  \sum_{1\leq k \leq d}   \|\widetilde g_k g_{k}\|_{L^1}  \left(\|R_k\|_{L^1_{t,x}} + \sigma^{-1} C_R\right)\\
&\lesssim  \sum_{1\leq k \leq d}     \left(\|R_k\|_{L^1_{t,x}} + \sigma^{-1} C_R  \right)\\
&\lesssim \|R\|_{L^1_{t,x}} + C_R\sigma^{-1},
\end{split}
\]
where the implicit constant does not depend on the parameter $\l$ or the given solution $(\rho , u ,R)$.

The estimates for the other two terms $\theta u$ and $\rho w $ follow from Lemma \ref{lemma:theta_estimate} and Lemma  \ref{lemma:w_estimate}. Indeed, H\"older's inequality gives
\[
\begin{split}
\| \theta u  \|_{L^1_{t,x}}
&\leq \| \theta \|_{L^1_{t,x}} \| u  \|_{L^\infty_{t,x}} \\
&\leq C_R \l^{-\gamma}, 
\end{split}
\]
and
\[
\begin{split}
\| \rho w  \|_{L^1_{t,x}}
&\leq  \|  w  \|_{L^1_{t,x}} \| \rho    \|_{L^\infty_{t,x}} \\
&\leq C_R \l^{-\gamma}.
\end{split}
\]

\end{proof}

\section{The new defect field \texorpdfstring{$R_1$}{R1} and its estimates} \label{sec:New_defect_field}
 
We continue with the proof of Proposition~\ref{prop:main_prop}.
Our next goal is to define a suitable defect field $R_1$ such that the new density $\rho_1$ and vector field $u_1$,
$$
\rho_1 : = \rho + \theta, \qquad u_1 : = u+ w,
$$
solve the continuity-defect equation
\begin{equation}\label{eq:new_equation_R_1}
\p_t \rho_1 + u_1 \cdot \nabla \rho_1  = \D R_1.
\end{equation}
The defect field will consists of three parts 
\[
R_1 =  R_{\Osc}  + R_{\Lin} + R_{\Cor},
\]
each solving the corresponding divergence equation:
\begin{align*}
\D R_{\Osc}  & =  \p_t \theta +\D (\theta_p  w  + R ),\\
\D R_{\Lin}  & =   \D (\theta  u + \rho w),  \\
\D R_{\Cor}  & = \D  (\theta_o  w    ).
\end{align*}

So we define the linear error $R_{\Lin}   =\theta  u + \rho w $ and the correction error $ R_{\Cor} =\theta_o  w $ in the usual way and the oscillation error $R_{\Osc}$ in the the following important lemma. Recall that $\mathcal{R} $ and $\mathcal{B} $ are the antidivergence operators defined in Appendix \ref{sec:append_tools}.

\subsection{Definition of the oscillation error}
\begin{lemma}[Space-time oscillations]\label{lemma:convex_integration_space_time}
The following identity holds
\[
\p_t \theta + \D (\theta_p  w  + R ) =\D \big( R_{\text{osc,x}}+ R_{\text{osc,t}}  + R_{\Acc} \big),
\]
where $R_{\text{osc,x}} $ is the spatial oscillation error
$$
R_{\text{osc,x}} = -   \sum_{1\leq k \leq d} \widetilde{g}_{k} g_{k} \mathcal{B} \Big(  \nabla   R_k , \big(  \bp_k \bw_k  -\fint_{\TT^d}  \bp_k \bw_k  \big)  \Big)  ,
$$
$R_{\text{osc,t}} $ is the temporal oscillation error
\[
R_{\text{osc,t}} =  \sigma^{-1}    \sum_{1\leq k \leq d} h_k   \p_t  R_k \ek .
\]
$R_{\Acc}$ is the acceleration   error
\[
R_{\Acc} =    - \sum_{1\leq k \leq d}\mathcal{B}\big(\p_t(\widetilde g_k R_k), \bp_k\big).
\]
 
\end{lemma}
 
\begin{proof}

By definition of $\theta_p$ and $ w $, using the disjointedness of supports of $\widetilde g_k$, $g_{k'}$ for $k \neq k'$, we obtain
\begin{align}\label{eq:convex_integration_space_time_1}
\D (\theta_p  w ) &=   - \sum_{1\leq k \leq d} \widetilde{g}_{k}  g_{k}  \D\big(  R_k  \bp_k \bw_k   \big).
\end{align}
Thanks to $ \D (     \bp_k \bw_k  ) =0 $, for each $k$,
\begin{align*}
  \D\big(  R_k  \bp_k \bw_k   \big)  & =   \nabla  R_k  \cdot \mathbb{P}_{\ne 0} \big( \bp_k \bw_k     \big)  + \D ( R_k \ek ).
\end{align*}
such that from \eqref{eq:convex_integration_space_time_1} we have the decomposition
\begin{align}\label{eq:convex_integration_space_time_11}
\p_t \theta + \D (\theta_p  w +R ) &=   O_1 +O_2 + O_3,  
\end{align}
with 
\begin{align*}
O_1 &:=  \p_t \mathbb{P}_{\ne 0} \theta_p,   \\
O_2 &:=  - \sum_{1\leq k \leq d}  \widetilde{g}_k g_k   \nabla  R_k  \cdot \mathbb{P}_{\ne 0} \big( \bp_k \bw_k     \big),\\
O_3 &:= \p_t \theta_o  - \sum_{1\leq k \leq d}  \widetilde{g}_kg_k \D ( R_k \ek )  + \D R.
\end{align*}

By the definitions of $R_{\Acc}$ and $\mathcal{R}$, the first term $O_1 = \D R_{\Acc}$ since $\bp_k$ has zero mean.

For the second term $O_2$, by definition of $\mathcal{B}$ and \eqref{eq:bilinear_B_identity} we observe that
\begin{equation}
\fint_{\TT^d}  R_k   \D ( \bp_k \bw_k )    +  \nabla R_k  \cdot \big( \bp_k \bw_k  -\fint_{\TT^d} \bp_k \bw_k  \big)
 = \D \mathcal{B} \Big( \nabla R_k  , \big( \bp_k \bw_k  -\fint_{\TT^d} \bp_k \bw_k  \big)  \Big), \label{eq:convex_integration_space_time_2}
\end{equation}
where the meaning of the vector-valued argument is that $\mathcal{B}$ is applied to each of its component. So \eqref{eq:convex_integration_space_time_2} implies $O_2 = \D R_{\text{osc,x}}$. 

Finally, for the last term $O_3$, by the definition of $\theta_o$ \eqref{eq:def_theta_o_2}, \eqref{eq:def_h_mu_t}, and \eqref{eq:g_theta_g_w_interaction}, 
\begin{align*}
\p_t \theta_o  &  =  \sigma^{-1}  \sum_{1\leq k \leq d} h'_k  \D (  R_k \ek)    + \sigma^{-1}  \sum_{1\leq k \leq d} h_k  \D (  \p_t  R_k \ek)       \\
&  = \Big(  \widetilde{g}_{\kappa}   g_{\kappa}  -  1 \Big)\sum_{1\leq k \leq d}  \D (  R_k \ek) + \sigma^{-1}  \sum_{1\leq k \leq d} h_k  \D ( \p_t   R_k \ek), 
\end{align*}
which concludes $O_3 =  \D  R_{\text{osc,t}}.$

\end{proof}

\subsection{Estimates of the new defect error}\label{subsec:estimate_R_1}
In the remainder of this section, we finish the proof of Proposition~\ref{prop:main_prop}. Given $\delta>0$, we will show that the sum of  $L^1_{t,x}$ norms of each error is less than $C_R \lambda^{-\gamma}$. This concludes the proof provided $\lambda$ is chosen large enough.

\subsubsection{$R_{\Acc}$ estimate:}
Taking advantage of the potential $ \mathbf{\Omega}_k$ as in \eqref{eq:Div_Omega_k_close_Phi_k}, we obtain
\begin{equation}\label{eq:proof_R_tem_1}
\begin{split}
 \| R_{\Acc} \|_{L^1_{t,x}} &=  \sigma^{-1}   \big\|  \mathcal{B}  \big( \p_t (\widetilde{g}_{k} R_k) , \D \mathbf{\Omega}_k    \big)     \big\|_{L^1_{t,x}}\\
  \scriptstyle{(\text{by Lemma \ref{lemma:cheapbound_B} })}      & \lesssim C_R \sigma^{-1}  \| \widetilde{g}_{k} \|_{W^{1,1}}  \big\| \mathcal{R}   \D\mathbf{\Omega}_k  \big\|_{  L^1} \\
  \scriptstyle{(\text{by  boundedness of $\mathcal{R}$ in $L^r$ })}    & \lesssim C_{R}  \sigma^{-1}   \| \widetilde{g}_{k} \|_{W^{1,1}} \|\mathbf{\Omega}_k \|_{ L^r}  \\
 \scriptstyle{(\text{by \eqref{eq:estimates_g} and \eqref{eq:BB_Bounds_Omega}})} & \lesssim C_{R}     \kappa^{\alpha}    \mu^{  -1 -\frac{ d -1}{r}} \\
\scriptstyle{(\text{by \eqref{eq:condition_acc_erro}})} &\lesssim C_R \lambda^{-\gamma}.
\end{split}
\end{equation}

\subsubsection{$R_{\text{osc,x}}$ estimate:}
 
By H\"older's inequality, Lemma~\ref{lemma:cheapbound_B}, and the bounds $\|\widetilde{g}_{k}  g_{k}\|_{L^\infty_t L^1} \sim 1$ and $\|   \bp_k \bw_k  \|_{L^1} \sim 1$ we obtain
\[
\begin{split}
\|    R_{\text{osc,x}}  \|_{L^1_{t,x}} &\leq  \sum_{1\leq k \leq d} \|    \widetilde{g}_{k}  g_{k} \|_{L^1}  \big\| \mathcal{B} \big( \nabla R_k  ,   \mathbb{P}_{\ne 0 } ( \bp_k \bw_k)  \big)    \big\|_{L^\infty_t L^1}\\ 
&\lesssim C_R\sum_{1\leq k \leq d}    \| \mathcal{R}    \mathbb{P}_{\ne 0 } ( \bp_k \bw_k)   \|_{L^\infty_t L^1}\\
&  \leq   C_R   \l^{-\gamma}.
\end{split}
\]

\subsubsection{$R_{\text{osc,t}}$ estimate:}

By \eqref{eq:def_h_mu_t_L_1},  
\[
\begin{split}
\|R_{\Osc,t} \|_{L^1_{t,x}} &= \Big\|\sigma^{-1} \sum_{1\leq k \leq d} h_k   \p_t  R_k \ek \Big\|_{L^1_{t,x}} \\
&\leq C_R \sigma^{-1}\sum_{1\leq k \leq d}   \|h_k \|_{L^1 } \\
&\leq  C_R\l^{-\gamma}.
\end{split}
\]

\subsubsection{$R_{\Lin}$ estimate:}

We start with H\"older's inequality
\begin{align*}
\big\|R_{\Lin}  \big \|_{L^1_{t,x}}   \leq\| \theta   \|_{L^1_{t,x}}    \| u \|_{L^\infty_{t,x}} +     \| \rho   \|_{L^\infty_{t,x}}    \| w \|_{L^1_{t,x}}.
\end{align*}
So it suffices to show $ \| \theta \|_{L^1_{t,x}} \leq C_R \l^{-\gamma } $ and $\| w  \|_{L^1_{t,x} } \leq C_R \l^{-\gamma } $. These follow from Lemma \ref{lemma:theta_estimate} and Lemma \ref{lemma:w_estimate} since $p\geq 1$.
 
\subsubsection{$R_{\Cor}$ estimate:}

By H\"older's inequality, 
\begin{align*}
\| R_{\Cor}   \|_{L^1_{t,x}} \leq \| \theta_o       \|_{L^\infty_{t,x}}\|  w    \|_{L^1_{t,x}}.
\end{align*}
Since $\| \theta_o       \|_{L^\infty_{t,x}} \leq C_R \l^{- \gamma } $ from its definition (or by  \eqref{eq:theta_o_estimates} from Lemma \ref{lemma:theta_estimate}), by Lemma \ref{lemma:w_estimate}  we also have $\| R_{\Cor}   \|_{L^1_{t,x}}  \leq C_R \l^{-\gamma}$.

\subsection{Conclusion of the proof of Proposition \ref{prop:main_prop}}
 The first point is proved in Section \ref{sec:temporal_BB_perturbations} while the second point in Section \ref{sec:Estimates_on_perturbation} provided $\l$ is sufficiently large. For the last point, \eqref{eq:support_of_R_1} follows from the definition of the new defect error $R_1$, and the estimate follows from the ones in Section \ref{subsec:estimate_R_1} by choosing $\l$ sufficiently large once again. Hence Proposition \ref{prop:main_prop} is proved.

\section{Extension to transport-diffusion equation}\label{sec:transport-diffusion}

In this section, we extend the main results to general transport-diffusion equations
\begin{equation}\label{eq:the_a-d-equation}
\begin{cases}
\partial_t \rho - L\rho + u \cdot \nabla  \rho =0 &\\
\rho|_{t=0} =\rho_0,
\end{cases}
\end{equation}
where $L$ is a given constant coefficient differential operator of order $k\in \NN$. Weak solutions to \eqref{eq:the_a-d-equation} can be defined analogously to \eqref{eq:def_weak_solutions} by the adjoint of $L$. The following nonuniquness result holds for \eqref{eq:the_a-d-equation}.
\begin{theorem}\label{thm:zero_viscosity_limit}

Let $d\geq 2$ and  $L$ be any constant coefficient differential operator of order $k \geq 1$. There exists a divergence free velocity vector field  $u\in L^1(0,T;W^{1,p}(\TT^d))$ for all $p< \infty $, such that the the uniqueness of \eqref{eq:the_a-d-equation} fails in the class 
$$ 
\rho \in \bigcap_{\substack{
  p<\infty  \\ k\in \NN}} L^p(0,T; C^k(\TT^d)).
$$

\end{theorem}
\begin{proof}
In view of the scheme adopted in this paper, we only need to check that  Proposition \ref{prop:main_prop} holds for \eqref{eq:the_a-d-equation}. It suffices to check the linear term $L \rho$ results in a small error
\[
R_{L}:= \mathcal{R} L\sum_{1\leq k \leq d} \widetilde g_k R_k \bp_k.
\]
Indeed, by $L^1$ boundedness of $\mathcal{R} $,
\[
\begin{split}
\|R_{L}\|_{L^1_{t,x}} &\lesssim  C_R  \sum_{1\leq k \leq d} \|\widetilde{g}_k\|_{L^1} \|  \bp_k\|_{L^\infty_t W^{k ,1}},
\end{split}
\]
where $k \geq 1$ is the order of the linear operator $L$. Since we only need to prove the results for $p$ large, we can assume $ k \leq p$ so that, as in the proof of Lemma \ref{lemma:theta_estimate},
\[
\begin{split}
\|R_{L}\|_{L^1_{t,x}}&\lesssim C_R  \kappa^{\alpha - 1 } (\sigma \mu)^{p} \leq C_R \lambda^{-\gamma}.
\end{split}
\]
Hence there is no additional constraint coming from the diffusion.  
\end{proof}

\appendix
\section{Standard tools in convex integration}\label{sec:append_tools}

In this section, we recall several technical results that are now standard in convex integration.
\subsection{Improved H\"older's inequality on \texorpdfstring{$\TT^d$}{Td}}
We recall the following result due to Modena and Sz\'ekelyhidi~\cite[Lemma 2.1]{MR3884855}, which was inspired by \cite[Lemma 3.7]{MR3898708}.

\begin{lemma}\label{lemma:improved_Holder}
Let $d \geq 2$, $r \in [1,\infty]$, and $a,f :\TT^d \to \RR$ be smooth functions. Then for every $\sigma \in \NN$,
\begin{equation}
\Big|   \|a f(\sigma \cdot ) \|_{r }  - \|a \|_{r} \| f \|_{r } \Big|\lesssim \sigma^{-\frac{1}{r}} \| a\|_{C^1} \| f \|_{ r }.
\end{equation}
\end{lemma}
Note that the error term on the right-hand side can be made arbitrarily small by increasing the oscillation factor $ \sigma$.

\subsection{Antidivergence operators \texorpdfstring{$\mathcal{R}$}{R} and \texorpdfstring{$\mathcal{B}$}{B}}
We will use the standard antidivergence operator $\Delta^{-1} \nabla$   on $\TT^d$, which will be denoted by $\mathcal{R}$.

It is well known that for any $f \in C^\infty  (\TT^d)$ there exists a  unique $ u \in C^\infty_0 (\TT^d)$ such that 
$$
\Delta u = f -\fint f.
$$ 
For any smooth scalar function $f \in C^\infty(\TT^d) $, the standard anti-divergence operator $ \mathcal{R} : C^\infty(\TT^d)  \to C^\infty_0(\TT^d, \RR^d ) $ can be defined as
$$
\mathcal{R} f:=  \Delta^{-1} \nabla f ,
$$
which satisfies
$$
\D ( \mathcal{R}  f )  = f  -\fint_{\TT^d}f\quad \text{for all $ f \in C^\infty(\TT^d)$}.
$$
It well-known (see for instance \cite[Lemma 2.2]{MR3884855}) that $ \mathcal{R}$ is bounded on Sobolev spaces $W^{k,p}(\TT^d)$ for all $k\in \NN$. And $ \mathcal{R}   \D$ is a Calderón–Zygmund operator,
$$
\| \mathcal{R} ( \D   u ) \|_r  \lesssim \| u \|_r   \quad \text{for all $ u \in C^\infty(\TT^d,\RR^d)$ and $1<  r <\infty$}.
$$
Recall the following useful fact about $\mathcal{R}$.
$$
\mathcal{R}f(\sigma \cdot ) = \sigma^{-1} \mathcal{R}f  \quad \text{for any $f \in C^\infty_0(\TT^d)$ and any positive $\sigma \in \NN$.}
$$

We will also use its bilinear counterpart $ \mathcal{B}: C^\infty(\TT^d) \times C^\infty(\TT^d) \to C^\infty (\TT^d,\RR^d) $ defined by
$$
\mathcal{B}(a,f) : = a \mathcal{R} f  - \mathcal{R}( \nabla a  \cdot \mathcal{R} f).
$$
It is easy to see that $\mathcal{B}$ is a left-inverse of the divergence,
\begin{equation}\label{eq:bilinear_B_identity}
\D (\mathcal{B}(a,f)  ) = af -\fint_{\TT^d} af \, dx \quad \text{provided that $f \in C^\infty_0(\TT^d)$, }
\end{equation}
which can be proved easily using integration by parts.
The following estimate is a direct consequence of the boundedness of $ \mathcal{R}$ on  Sobolev spaces $W^{k,p}(\TT^d)$.

\begin{lemma}\label{lemma:cheapbound_B}
Let $d \geq 2$ and $1\leq r  \leq \infty$. Then for any $a,f \in C^\infty(\TT^d)$
\begin{align*}
\| \mathcal{B}(a,f)\|_r \lesssim \| a\|_{C^1} \|\mathcal{R}  f\|_r .
\end{align*}
\end{lemma}

\bibliographystyle{alpha}
\bibliography{nonuniqueness_transport}

\end{document}